\documentclass[10pt]{article}

\usepackage{geometry}
\usepackage{hyperref}
\usepackage{amsmath,amssymb,amsthm}
\usepackage[active]{srcltx}
\usepackage[dvips]{graphicx}
\usepackage{mathrsfs}
\newtheorem{thm}{\textbf{Theorem}}[section]
\newtheorem{lem}[thm]{\textbf{Lemma}}
\newtheorem{prop}[thm]{\textbf{Proposition}}
\theoremstyle{remark}
\newtheorem{rem}[thm]{\textbf{Remark}}
\newtheorem{cor}[thm]{\textbf{Corollary}}

\theoremstyle{definition}

\newtheoremstyle{Claim}{}{}{\itshape}{}{\itshape\bfseries}{:}{ }{#1}
\theoremstyle{Claim}

\newcommand{\Rset}{\mathbb{R}}
\newcommand{\RsetN}{ {\mathbb{R}^N} }
\newcommand{\Tn}{ {\mathbb{T}^N} }
\newcommand{\intRn}{\int_{\Rset^N}}
\newcommand{\intTn}{\int_\Tn}

\newcommand{\J}{\mathcal{E}}
\newcommand{\F}{\mathscr{F}}

\begin{document}

\title{Stationary focusing Mean Field Games}
\author{Marco Cirant}

\maketitle

\begin{abstract} We consider stationary viscous Mean-Field Games systems in the case of local, decreasing and unbounded coupling. These systems arise in ergodic mean-field game theory, and describe Nash equilibria of games with a large number of agents aiming at aggregation. We show how the dimension of the state space, the behavior of the coupling and the Hamiltonian at infinity affect the existence and non-existence of regular solutions. Our approach relies on the study of Sobolev regularity of the invariant measure and a blow-up procedure which is calibrated on the scaling properties of the system. In very special cases we observe uniqueness of solutions. Finally, we apply our methods to obtain new existence results for MFG systems with competition, namely when the coupling is local and increasing.
\end{abstract}

\noindent
{\footnotesize \textbf{AMS-Subject Classification}}. {\footnotesize 35J47, 49N70, 35B33.}\\
{\footnotesize \textbf{Keywords}}. {\footnotesize Concentration, critical exponent, Pohozaev identity, Gagliardo-Nirenberg inequality.}

\section{Introduction}

In this paper we investigate stationary viscous Mean Field Games (MFG) systems of the form
\begin{equation}\label{MFG}
\begin{cases}
- \Delta u(x) + H(\nabla u(x)) + \lambda = V(x) - f(m(x))  \\
- \Delta m(x) -{\rm div}(\nabla H(\nabla u(x)) \, m(x)) = 0  & \text{in $\Omega$}  \\
\int_\Omega m \, dx = 1, \, m > 0,
\end{cases}
\end{equation}
where the function $f$ is non-negative, and $\Omega \subseteq \RsetN$. In particular, we have in mind nonlinearities of the form
\begin{equation}\label{Hf_partic}
H(p) = \frac{1}{\gamma} |p|^{\gamma}, \quad \gamma > 1, \qquad \text{and} \qquad f(m) = C_f m^{\alpha}, \quad C_f, \alpha > 0.
\end{equation}

The MFG system \eqref{MFG} captures Nash equilibria of infinite-horizon games with a large number of indistinguishable rational agents, who seek to optimize an individual utility. In particular, a typical agent of the game, distributed in the long-time regime with invariant density $m$, pays a cost which depends on his velocity and $V(x) - f(m(x))$, where $x \in \Omega$ is his own state. While $V$ can be considered as a fixed potential, $-f(m(x))$ depends on the distribution $m$ itself: this term realizes the coupling between the individual and the overall population, and the coupling between the Hamilton-Jacobi-Bellman and the Kolmogorov equations in \eqref{MFG} at the PDE level. In this model, every agent is also subject to a Brownian noise.

From the game point of view, if $-f(\cdot)$ is an increasing function every individual aims at avoiding regions where the population is highly distributed. This class of MFG systems has received a considerable attention in the last years, starting from the seminal works \cite{HCM1, HCM2, jeux1, jeux2, LasryLions}. It has been shown that \eqref{MFG} corresponds to the long-time limit of non-stationary MFG (see \cite{Cardalialong, Cardalialong2}). For a description and up-to-date developments on MFG we refer the reader to \cite{MR3134900, MR3195844, Cardalianotes, LionsVideo}.

Here, we address the problem of existence of solutions to \emph{focusing} MFG systems, namely when the coupling $-f(\cdot)$ in \eqref{MFG} is a \emph{decreasing} function; that is, we assume that every player of the game is attracted by regions where the population is highly distributed. As far as we know, this setting has not yet been expolored systematically; a study in the quadratic and linear-quadratic case has been carried out in \cite{BardiPriuli, gueant09, gueant12}. A serious technical difficulty here comes from the fact that the couplings $-f$ we consider are not bounded from below. The particular coupling $-f(m) = \log m$, having the opposite monotonicity but lacking of boundedness from below, has been treated for example in \cite{PV, GomesPimentel} (see also references therein). Still, we consider this framework \emph{defocusing}, as $ \log(\cdot)$ is increasing. We stress that in our case the lack of increasing monotonicity of $-f$ cuts off from a large class of general approaches that have been developed in the MFG literature, for example in \cite{CardaliaguetGraber, PorrettaWeak, GomesFerreira}. 

As for the existence of solutions, we set our problem on the flat torus, i.e. $\Omega = \Tn$; in this setting we avoid boundary issues, and exploit compactness of the state space. Our focus is to obtain solutions $(u, \lambda, m)$ of \eqref{MFG} such that $\nabla u$ and $m$ are (at least) \emph{bounded}. We believe that this requirement is meaningful from the point of view of the game: $-\nabla H(\nabla u)$ provides (formally) the optimal strategy of an average player in feedback form, and bounded $\nabla u$ guarantees boundedness of the optimal velocity (and therefore an agent does not have to move with infinite velocity). Moreover, boundedness of $m$ is a crucial point in our analysis: here, $m$ has an intrinsic tendency to concentrate and hence to develop singularities. We have to carefully examine the delicate interplay between the Brownian motion, that has a smoothing effect on the distribution, and the focusing behavior of the distribution itself, which is not an issue in the defocusing case.

At the PDE level, even a-priori bounds on the ergodic constant $\lambda$ are not obvious, as well as $L^1$ bounds on $m^{\alpha+1}$ and $|\nabla u|^\gamma m$, which are related to the ``energy'' of the system \eqref{MFG} (see Remark \ref{rem:var}); those bounds are somehow ``for free'' in the defocusing case. We clarify this aspect of the problem through Proposition \ref{prop_energy} and Corollary \ref{cor_energy}, which establish a regularity result for Kolmogorov equations (of independent interest), providing estimates on $m^{\alpha+1}$ with respect to $|\nabla u|^\gamma m$. In particular, we observe that there are two ``critical'' exponents
\[
\alpha_1 = \frac{\gamma'}{N} \quad < \quad \alpha_2 = \frac{\gamma'}{N-\gamma'},
\]
(where $\gamma' = \gamma/(\gamma-1)$ is the usual conjugate exponent) that give rise to three different qualitative behaviors for \eqref{MFG}:  if $\alpha \in (0, \alpha_1)$, then $\lambda$ and the ``energy'' are bounded a-priori (and hence a solution exists), while if $\alpha \in [\alpha_1, \alpha_2)$ estimates can be obtained only under additional conditions on the coupling. If $\alpha \in (\alpha_2, \infty)$ we will see that the problem may not possess bounded solutions. Such critical exponents can be better understood by having a look at the variational formulation of \eqref{MFG}, that is discussed in Remark \ref{rem:var}; basically, $\alpha_2$ is related to the Sobolev critical exponent, while $\alpha_1$ comes from the Gagliardo-Nirenberg inequality and the $L^1$ constraint on $m$.

\smallskip

We suppose that $H \in C^2(\RsetN \setminus \{0\})$ and that there exist $C_H > 0$, $\gamma > 1$ such that
\begin{equation}\label{Hass}
\begin{split}
\bullet \quad & C_H^{-1}|p|^\gamma-C_H \le H(p) \le C_H(|p|^{\gamma} + 1), \\
\bullet \quad & |\nabla H(p)| \le  C_H(|p|^{\gamma-1} + 1),\\
\bullet \quad & \nabla H(p) \cdot p - H(p) \ge C_H^{-1} |p|^\gamma - C_H,
\end{split}
\end{equation}
for all $p \in \RsetN$. Moreover, $f \in C^1((0, \infty))$, $V \in C^1(\Omega)$ and there exist $\alpha, C_f, C_V > 0$ such that
\begin{align}
& 0 \le f(m) \le C_f(m^\alpha + 1) \quad \text{for all $m \ge 0$}, \label{fass} \\
& 0 \le V(x) \le C_V \quad \text{for all $x \in \RsetN$}. \label{Vass1}
\end{align}
Of course, the model nonlinearities \eqref{Hf_partic} satisfy \eqref{Hass}, \eqref{fass}. Note that we are not requiring monotonicity of $f$ in general; our methods really rely on the behavior of $f$ at infinity.


The first main result of this paper, regarding existence of solutions in the ``subcritical'' case $\alpha < {\gamma'}/({N-\gamma'})$, is stated in the following

\begin{thm}\label{thm_existence}
Let $\Omega = \Tn$, $H, f, V$ be such that \eqref{Hass}, \eqref{fass}, \eqref{Vass1} hold. If
\[
\alpha < \gamma'/N,
\]
then there exists a solution $(u, \lambda, m) \in C^2(\Tn) \times \Rset \times W^{1,p}(\Tn)$, for all $p > 1$, of \eqref{MFG}. Else, if
\begin{equation}
\gamma'/N \le \alpha < \begin{cases}
\gamma'/(N-\gamma') & \text{if $\gamma' < N$} \\
\infty & \text{if $\gamma' \ge N$,}
\end{cases}
\end{equation}
a solution $(u, \lambda, m)$ of \eqref{MFG} exists under the additional condition that $C_f$ (in \eqref{fass}) is small enough.
\end{thm}

The proof of Theorem \ref{thm_existence} relies on an approximating procedure. Proposition \ref{prop_energy} provides the main estimates for $\lambda$ and the $L^1$ norm of $m^{\alpha +1}$ and $|\nabla u|^\gamma m$ (see \eqref{EI1}-\eqref{EI2}). The standard tool of Schauder's fixed point theorem implies the existence of a solution $(u_k, \lambda_k, m_k)$ of a suitable ``regularized'' version of \eqref{MFG}. It is crucial to produce  approximate solutions whose energy is uniformly bounded with respect to $k$. Still, such estimates are not powerful enough to pass to the limit $k \to \infty$ in the approximating problem \eqref{mkdef}. The new idea presented here, which is used to obtain uniform $L^\infty$ bounds for $m_k$, leans on a blow-up method and exploits the scaling properties of \eqref{MFG}. We observe that our blow-up procedure works if $\alpha$ is below the critical exponent $\alpha_2$, which is the same exponent that appears in Proposition \ref{prop_energy}. This key bound enables us to use the machinery of Schauder and classical elliptic estimates, and to prove the existence of a regular solution of \eqref{MFG}.

\begin{rem} We want Hamiltonians of the form \eqref{Hf_partic} to fit into our theory, and therefore we cannot expect to obtain \emph{classical} solutions in general: even if $u$ is $C^\infty$, $\nabla H(\nabla u)$ is just an Holder function if $\gamma < 2$, so $m$ may not be a solution in the classical sense of the Kolmogorov equation in \eqref{MFG}. We will look for solutions $(u,m) \in C^2 \times W^{1,p}$, and consider them ``regular'', at least if compared with \emph{weak} solutions that are obtained by other methods in the defocusing case (see \cite{PorrettaWeak} and \cite{GomesFerreira}). Still, if $\gamma \ge 2$, or if $\nabla H$ is suitably regular, solutions will be a-posteriori smooth in view of $L^\infty$-boundedness of $m$, $\nabla u$.
\end{rem}

The second part of this work focuses on the ``supercritical'' case $\alpha \ge \gamma'/(N-\gamma')$. In this regime, the fast decay of the coupling $-m^\alpha$ might not be compensated by viscosity, leading to spike formation in the distribution $m$. This point can be made more clear by considering the variational formulation of \eqref{MFG} (see again Remark \ref{rem:var}), where ``compactness'' of the problem is somehow lost.

We shift our attention to the whole space, that is $\Omega = \RsetN$. The flat torus is indeed not suitable anymore for studying non-existence and concentration phenomena, as $(u, m) \equiv (0, 1)$ is always a solution of \eqref{MFG} if $\Omega = \Tn$, for all $\alpha > 1$. An unbounded state space boils down the possibility of having the constant solution. We will also set $V \equiv 0$, for simplicity, and $H$ will be of the form \eqref{Hf_partic}. Moreover, we will consider classical solutions $u$ satisfying the following condition at infinity (see the discussion in Remark \ref{rem:infinity}).
\begin{equation}\label{uinfty}
\text{$u \to +\infty$ as $|x| \to \infty$}, \quad \text{and} \quad \text{$\exists c, \eta > 0$ s.t. $|\nabla u(x)| \le c(|x|^\eta + 1), \forall x \in \RsetN$.}
\end{equation}

In order to understand \eqref{MFG} in the case $\alpha \ge \gamma'/(N-\gamma')$ we multiply the equations in \eqref{MFG} by $x \cdot \nabla m$ and $x \cdot \nabla u$, and obtain new integral identities (see Proposition \ref{pohozaev}). Let us set \[F(m) := \int_0^m f(s) \, ds,\]
and suppose that the following condition holds true:
\begin{equation}\label{fass2}
(N-\gamma') f(m) m - N F(m) > 0 \quad \text{for all $m > 0$.}
\end{equation}
Note that if $f$ is of the form \eqref{Hf_partic} and $\alpha > \gamma'/(N-\gamma')$, then \eqref{fass2} holds.

The second main result of this paper, regarding non-existence of ``regular'' solutions, is stated as follows.

\begin{thm}\label{thm_nonex} Let $\Omega = \RsetN$, suppose that $H$ is of the form \[ H(p) = \frac{1}{\gamma} |p|^\gamma, \quad \gamma > 1, \] that \eqref{fass2} holds, and $V \equiv 0$. Then, \eqref{MFG} has no solutions $(u, \lambda, m) \in C^2(\RsetN) \times \Rset \times W^{1,1}(\RsetN)$ satisfying
\begin{equation}\label{finiteness}
\intRn |\nabla u|^\gamma m \, dx < \infty, \quad \intRn m^{\alpha+1} \, dx < \infty, \quad \intRn |\nabla u |\, |\nabla m| \, dx < \infty
\end{equation}
and \eqref{uinfty}.
\end{thm}

Theorems \ref{thm_existence} and \ref{thm_nonex} give a rather precise picture of the problem of existence of regular solution for \eqref{MFG}, showing that criticality phenomena arise. We have chosen the state space $\RsetN$ in the non-existence part of the work in order to avoid some technical difficulties, but we believe that existence of solutions can be obtained under the same assumptions on the growth of $f$ of Theorem \ref{thm_existence}. This and other aspects of MFG systems on the whole space will be matter of future work. The critical case $\alpha = \gamma'/(N-\gamma')$ is also not covered by our theorems. 

We conclude our study of focusing MFG by considering uniqueness of solutions. Even though the standard uniqueness condition by Lasry-Lions (see, for example, \cite{LasryLions}) is violated, we can lean on some results on uniqueness of ground states of nonlinear Schr\"odinger equations and show that solutions of a particular class of quadratic MFG systems are unique. The main difficulty here is that $u,m$ might not be unique even if the ergodic constant $\lambda$ is fixed, and $\lambda$ is itself an unknown of the problem; this issue is circumvented by exploiting invariances of the system. Our considerations on this aspect of the problem are collected in Section \ref{s:unique}.

\medskip

We mention that the methods developed here for focusing MFG can be implemented also to study defocusing MFG systems, that is
\begin{equation}\label{MFGde}
\begin{cases}
- \Delta u(x) + H(\nabla u(x)) + \lambda = V(x) + f(m(x))  \\
- \Delta m(x) -{\rm div}(\nabla H(\nabla u(x)) \, m(x)) = 0  & \text{in $\Omega$}  \\
\int_\Omega m \, dx = 1, \, m > 0,
\end{cases}
\end{equation}

Following the same lines of the proof of Theorem \ref{thm_existence}, we are able to prove the following
\begin{thm}\label{thm_existence_def}
Let $\Omega = \Tn$, $H, f, V$ be such that \eqref{Hass}, \eqref{fass}, \eqref{Vass1} hold. If
\begin{equation*}
\alpha < \begin{cases}
\gamma'/(N-\gamma') & \text{if $\gamma' < N$} \\
\infty & \text{if $\gamma' \ge N$,}
\end{cases}
\end{equation*}
then there exists a solution $(u, \lambda, m) \in C^2(\Tn) \times \Rset \times W^{1,p}(\Tn)$, for all $p > 1$, of \eqref{MFGde}.
\end{thm}

Existence of smooth solutions for \eqref{MFGde} is still an open problem for general power-like nonlinearities, and known results require conditions on the exponents $\alpha, \gamma$; see \cite{MR3333058, GomesPatrizi, GomesMorgado, PV}. As far as we know, the best available results are contained in \cite{PV}, and require $\alpha \le \gamma'/N$. We are able to improve this condition in Theorem \ref{thm_existence_def}. We stress that the critical exponent $\gamma'/(N-\gamma')$ does not seem to be optimal in the defocusing case (if $\gamma =2$, smooth solutions exist for all $\alpha \ge 0$ and for all $N$), but is fundamental in the focusing case, where concentration of solutions is intrinsically likely to arise.

This paper is structured as follows. In Section \ref{s:ex} we will discuss the existence of solutions for \eqref{MFG} in the flat torus. Section \ref{s:nonex} will be devoted to the non-existence proof in the whole space, by means of Pohozaev identities. In Section \ref{s:unique} we will collect some observations on uniqueness of solutions in some special cases, while in Appendix \ref{a:defoc} we will prove the existence of regular solutions in the defocusing and ``subcritical'' case.

{\bf Notation.} Throughout the paper, we will refer to the HJB equation and to the Kolmogorov equation in \eqref{MFG} as the first and second equation of the system, respectively. For all $R > 0$, $x \in \RsetN$, $B_R(x) := \{y \in \RsetN: |y-x| < R\}$, $B_R := B_R(0)$. If $\Omega$ is a smooth domain of $\RsetN$, $\nu$ will denote the outward normal vector field at $\partial \Omega$.  Finally, $C, C_1, C_2,\ldots$ will be (positive) constants we need not to specify.

\section{Existence of solutions}\label{s:ex}

This section is devoted to the proof of existence of solutions to \eqref{MFG}, where the state space $\Omega$ is the flat torus $\Tn$. Unless otherwise specified, $L^p$ and $W^{1,p}$ norms will be intended on $\Tn$, that is $\| \cdot \|_{L^p} = \| \cdot \|_{L^p(\Tn)}$. We will always assume that $H$ satisfies \eqref{Hass} and $\alpha$ of assumption \eqref{fass} will be such that
\begin{equation}\label{critical_alpha}
\alpha < \begin{cases}
\gamma'/(N-\gamma') & \text{if $\gamma' < N$} \\
\infty & \text{if $\gamma' \ge N$.}
\end{cases}
\end{equation}

We start by recalling some classical results on regularity of solutions of uniformly elliptic and HJB equations, that will be used in the sequel.

\begin{prop}\label{ell_regularity} Let $\Omega \subseteq \RsetN$, $p > 1$ and $u \in L^p(\Omega)$ be such that
\[
\left| \int_\Omega u \, \Delta \varphi \, dx \right| \le K\|\varphi\|_{W^{1,p'}(\Omega)} \quad \text{for all $\varphi \in C^\infty_c(\Omega)$}
\]
for some $K > 0$. Then, for all $\Omega' \subset \subset \Omega$, $u \in W^{1,p} (\Omega')$ and there exists $C > 0$, depending on $\Omega$, $\Omega'$ and $p$ such that
\[
\|u\|_{W^{1,p}(\Omega')} \le C(K + \|u \|_{L^p(\Omega)}).
\]
\end{prop}

\begin{proof} See, for example, \cite[Theorem 6.1]{agmon}. \end{proof}

\begin{prop}\label{hjb_regularity} Suppose that $\Omega \subseteq \RsetN$, $H$ satisfies \eqref{Hass} and $u \in C^2(\RsetN)$ is such that
\[
\left| -\Delta u + H(\nabla u) \right| \le K \quad \text{in $\Omega$}.
\]
Then, for all $\Omega' \subset \subset \Omega$, $q > 1$,
\[
\|\nabla u\|_{L^q(\Omega')} \le C,
\]
where $C > 0$ depends on $\Omega'$, $\Omega$, $K$, $q$ and $C_H, \gamma$ in \eqref{Hass}.
\end{prop}

\begin{proof} This estimate, which relies on the Bernstein method, has been proved in \cite[Theorem A.1]{lasry1989nonlinear} when $H$ is of the form \eqref{Hf_partic}. If one looks carefully at the proof, it is possible to carry out the same procedure also for perturbations of that Hamiltonians, namely when $H$ satisfies \eqref{Hass}. A detailed proof in this case can be found in \cite{PV}, where the HJB equation is set on the flat torus $\Tn$. \end{proof}

In the following proposition we prove the inequalities \eqref{EI1A}-\eqref{EI2A}, by a delicate combination of Sobolev embeddings and elliptic regularity. The idea of the first part of the proof comes from \cite{Metafune}; here we compute explicitly some key exponents. As for \eqref{EI2A}, it follows by an appropriate use of the Gagliardo-Nirenberg inequality.

\begin{prop}\label{prop_energy} Suppose that $m \in W^{1,2}(\Tn)$, $A \in L^{\infty}(\Tn)$ solve (weakly\footnote{that is: $\intTn \nabla m \cdot \nabla \varphi - m A \cdot \nabla \varphi = 0$ for all $\varphi \in C^{\infty}(\Tn)$.})
\begin{equation}\label{kolmo}
- \Delta m(x) + {\rm div}(A(x) \, m(x)) = 0 \quad \text{in $\Tn$}, \qquad \int_\Tn m \, dx = 1.
\end{equation}

Then, for all $\beta > 1$ such that
\begin{equation}\label{beta}
\beta < \begin{cases}
1 + \gamma'/(N-\gamma') & \text{if $\gamma' < N$} \\
\infty & \text{if $\gamma' \ge N$,}
\end{cases}
\end{equation}
there exists $C > 0$ and $\delta > 1$, depending on $\beta$, $N$ and $\gamma'$ such that
\begin{equation}\label{EI1A}
\|m\|^\delta_{L^\beta} \le C\left(\intTn |A|^{\gamma'} m  \, dx  + 1\right).
\end{equation}

Moreover, for all $1 < \beta < 1 + \gamma'/N$ it holds true that
\begin{equation}\label{EI2A}
\|m\|^{\delta \beta}_{L^\beta} \le C\left(\intTn |A|^{\gamma'} m \, dx  + 1\right).
\end{equation}

\end{prop}

\begin{proof} Set
\begin{equation}\label{Edef}
E := \intRn |A|^{\gamma'} m \, dx.
\end{equation}
The inequality \eqref{EI1A} will be proved firstly. One observes that $m$ solves \eqref{kolmo}, so
\[
\left| \intTn m \, \Delta \varphi \, dx \right| \le \intTn |A| m^{1/\gamma'} \, m^{1- 1/\gamma'} \, |\nabla \varphi| \, dx,
\]
and an application of the Holder inequality provides
\[
\left| \intTn m \, \Delta \varphi \, dx \right| \le E^{1/\gamma'}\|m\|^{1/\gamma}_{L^{\beta}} \|\nabla \varphi\|_{L^{r'}},
\]
where the following equality holds
\[
\qquad \frac{1}{r} = \frac{1}{\gamma'} + \left(1 - \frac{1}{\gamma'}\right)\frac{1}{\beta}.
\]
In view of Proposition \ref{ell_regularity} one argues that
\[
\|m\|_{W^{1,r}} \le C_r (E^{1/\gamma'}\|m\|^{1/\gamma}_{L^{\beta}} + \|m\|_{L^r}).
\]
By interpolation $\|m\|_{L^r} \le \|m\|^{1/\gamma'}_{L^1} \|m\|^{1/\gamma}_{L^{\beta}}$, so
\begin{equation}\label{eq11}
\|m\|_{W^{1,r}} \le C_1 \|m\|^{1/\gamma}_{L^{\beta}} (E^{1/\gamma'} + 1),
\end{equation}
and by standard Sobolev embedding one has
\begin{equation}\label{eq12}
\|m\|_{L^{\eta}}  \le C_2 \|m\|^{1/\gamma}_{L^{\beta}} (E^{1/\gamma'} + 1),
\end{equation}
where
\[
\frac{1}{\eta} = \frac{1}{r} - \frac{1}{N} = \frac{1}{\gamma'} - \frac{1}{N} + \left(1 - \frac{1}{\gamma'}\right)\frac{1}{\beta}.
\]
Note that $1 < r < N$, if $\gamma' < N$. One verifies also that $\eta > \beta$, because $\beta < N/(N-\gamma')$. Therefore, again by interpolating between $L^1$ and $L^{\eta}$ there exists $0 < \theta < 1$ such that
$\|m\|_{L^{\beta}} \le \|m\|_{L^{\eta}}^\theta$, and it follows that
\[
\|m\|^{1/\theta}_{L^{\beta}}  \le C_2 \|m\|^{1- 1/\gamma'}_{L^{\beta}} (E^{1/\gamma'} + 1).
\]
Hence one has \eqref{EI1A} by setting $\delta = 1 + (1/\theta-1) \gamma'$.

If $\gamma' \ge N$, then \eqref{eq12} is satisfied for all $\eta \ge 1$, so the claimed inequality \eqref{EI1} again follows by interpolation between $L^1$ and $L^{\eta}$, with $\eta$ large enough.

In order to prove \eqref{EI2A}, the Gagliardo-Nirenberg inequality will be used instead of the Sobolev inequality. In particular, one has
\[
\|m\|_{L^{\eta}}  \le C_3 ( \| \nabla m \|_{L^r}^{N/(N+1)} \| m \|_{L^1}^{1/(N+1)} + \| m \|_{L^r}^{N/(N+1)}),
\]
where
\[
\frac{1}{\eta} = \left(\frac{1}{r} - \frac{1}{N}\right) \frac{N}{N+1} + 1 -  \frac{N}{N+1} = \frac{N}{N+1}\left[\frac{1}{\gamma'} + \left(1 - \frac{1}{\gamma'}\right)\frac{1}{\beta}\right],
\]
so, by plugging \eqref{eq11} into the last inequality,
\[
\|m\|^{(N+1)/N}_{L^{\eta}}  \le C_4 \|m\|^{1/\gamma}_{L^{\beta}} (E^{1/\gamma'} + 1).
\]
It is now crucial to observe that $\eta > \beta$ because $\beta < 1 + \gamma'/N$. Arguing by interpolation as before one has
\[
\|m\|^{(N+1)/(\theta N)}_{L^{\beta}}  \le C_5 \|m\|^{1- 1/\gamma'}_{L^{\beta}} (E^{1/\gamma'} + 1),
\]
which implies
\[
\|m\|^{1 + \gamma' (N+1)/(\theta N) - \gamma'}_{L^{\beta}}  \le C_6(E + 1).
\]
The inequality \eqref{EI2} then follows because $1 + \gamma' (N+1)/(\theta N) - \gamma' \ge \gamma'/N + 1 > \beta$.

\end{proof}

\begin{cor}\label{cor_energy} Suppose that $(u, m) \in C^2(\Tn) \times W^{1,2}(\Tn)$ solves the Kolmogorov equation in \eqref{MFG}. Then, for all $\beta$ satisfying \eqref{beta},
\begin{equation}\label{EI1}
\|m\|^\delta_{L^\beta} \le C\left(\intTn |\nabla u|^\gamma m \, dx  + 1\right),
\end{equation}
and for all $1 < \beta < 1 + \gamma'/N$ it also holds true that
\begin{equation}\label{EI2}
\|m\|^{\delta \beta}_{L^\beta} \le C\left(\intTn |\nabla u|^\gamma m \, dx  + 1\right).
\end{equation}
\end{cor}

The corollary clearly follows from \eqref{EI1A}, \eqref{EI2A}, with $A = -\nabla H(\nabla u)$ and the assumptions on the Hamiltonian \eqref{Hass} ($C$ will also depend on $C_H$).

\begin{rem}\label{rem_energy} If one goes back to the proof of Proposition \ref{prop_energy}, boundedness of $\int |A|^{\gamma'} m$ implies uniform boundedness of $\|m\|_{L^\beta}$ for all $\beta$ satisfying \eqref{beta}, but also that any solution of the Kolmogorov equation $m_k$ such that $\int |A_k|^{\gamma'} m_k \le C$ converges (up to subsequences) in $L^\beta(\Tn)$. This is true because Sobolev embeddings are compact (see \eqref{eq11}-\eqref{eq12}).
\end{rem}

We are now ready to prove the main existence result, which will be obtained through a fixed point/approximation scheme. Several lemmas will provide the required regularity of the approximating problem, and will justify the limiting procedure.

\begin{proof}[Proof of Theorem \ref{thm_existence}]

Let $k$ be a positive integer, $\psi$ be a radial mollifier (i.e. $\psi \ge 0$ and $\intRn \psi = 1$) and
\[
\psi_k := k^N \psi(kx), \quad \forall k \ge 1, x \in \RsetN.
\]
Of course $\intRn \psi_k = 1$ for all $k$. Let $(u_k, \lambda_k, m_k) \in C^2(\Tn) \times \Rset \times W^{1,p}(\Tn)$ be the defined by
\begin{equation}\label{mkdef}
\begin{cases}
- \Delta u_k + H(\nabla u_k) + \lambda_k = V(x) - f(m_{k} \star \psi_k(x))  \\
- \Delta m_k -{\rm div}(\nabla H(\nabla u_k) \, m_k) = 0  & \text{in $\Tn$},  \\
\intTn m \, dx = 1, \, m > 0,
\end{cases}
\end{equation}
where $\star$ is the standard convolution operator.

\begin{lem}\label{lemma_welldef} If $\alpha < \gamma'/N$,
\begin{align}
\bullet \quad &\text{the triple $(u_k, \lambda_k, m_k)$ is well-defined for all $k \in \mathbb{N}$,} \label{welldef} \\
\bullet \quad & \text{$\exists C > 0$ s.t. $\forall k \ge 1$,} \quad |\lambda_k| \le C, \quad \intTn m_k^{\alpha+1} \, dx \le C, \quad \intTn |\nabla u_k|^{\gamma} m_k \, dx \le C. \label{energybound}
\end{align}

Otherwise, if $\alpha \ge \gamma'/N$, then \eqref{welldef} and \eqref{energybound} hold under the additional condition that $C_f$ is small enough.
\end{lem}

\begin{proof} In order to prove that a solution $(u_k, \lambda_k, m_k)$ of \eqref{mkdef} exists, one looks for fixed points of the map $\F = \F_k : \mu \mapsto m$, defined by $\mu \mapsto (v, \lambda) \mapsto m$, where $(v, \lambda)$ is a solution of
\begin{equation}\label{FHJB}
- \Delta v + H(\nabla v) + \lambda = V(x) - f(\mu \star \psi_k), \quad \text{in $\Tn$},
\end{equation}
and $m$ is the invariant distribution solving
\begin{equation}\label{FKolmo}
- \Delta m -{\rm div}(\nabla H(\nabla v) \, m) = 0 , \quad \text{in $\Tn$}, \quad \intTn m \, dx = 1, \, m > 0.
\end{equation}

Let
\[
\mathcal{K}_\xi := \left\{ m \in W^{1,p}(\Tn) : \intTn m \, dx = 1, \, \intTn m^{\alpha + 1} \, dx \le \xi, \, m \ge 0 \right\},
\]
for $\xi > 0$ ($p > N$ is chosen). The fact that $\F$ is well-defined, continuous and compact on $C(\Tn) \supset \mathcal{K}_\xi$ is standard (see, for example, \cite{MR3333058}). Moreover, one has the following

\begin{lem} If $\alpha < \gamma'/N$, then there exists $\xi > 0$ such that $\F$ maps $\mathcal{K}_\xi$ into itself. Otherwise, if $\alpha \ge \gamma'/N$, the existence of $\xi$ such that $\F$ maps $\mathcal{K}_\xi$ into itself holds under the additional condition that $C_f$ is small enough.
\end{lem} 

\begin{proof} Let $\mu \in \mathcal{K}_\xi$, and $v, \lambda, m$ be defined above. Firstly, $\lambda$ is bounded from above, uniformly with respect to $\mu$ and $k$. Indeed, integrating the HJB equation \eqref{FHJB} on $\Tn$,
\[
-C_H + \lambda \le \intTn H(\nabla v) \, dx + \lambda = \intTn V(x) \, dx - \intTn f(\mu \star \psi_k(x)) \, dx \le C_V,
\]
by \eqref{Hass} and \eqref{Vass1}.

In order to prove that $\int m^{\alpha + 1} \, dx \le \xi$, one starts by multiplying \eqref{FHJB} by $m$, and \eqref{FKolmo} by $v$; integrating by parts leads to
\[
\intTn (\nabla H(\nabla v) \cdot \nabla v -H(\nabla v) ) m \, dx + \intTn V \, m \, dx = \lambda + \intTn f(\mu \star \psi_k) m \, dx.
\]
In view of \eqref{Hass} and \eqref{fass}, it holds true that
\begin{multline}\label{eq20}
C_H^{-1} \intTn |\nabla v|^\gamma m \, dx + \intTn V \, m \, dx \le \\ \lambda +  C_H + \intTn f(\mu \star \psi_k) m \, dx \le C_1 + C_f \|(\mu \star \psi_k)^\alpha m\|_{L^1},
\end{multline}
hence by Holder inequality and standard properties of the convolution
\begin{equation}\label{eq21}
 \intTn |\nabla v|^\gamma m \, dx \le C_2 + C_H C_f \|\mu\|^\alpha_{L^{\alpha + 1}} \|m\|_{L^{\alpha+1}}.
\end{equation}

\underline{Case 1}: $\alpha < \gamma'/N$. The inequality \eqref{EI2} applies, and one has
\[
\|m\|^{\delta(\alpha + 1)}_{L^{\alpha +1}} \le C_3(\| \mu \|^\alpha_{L^{\alpha + 1}} \|m\|_{L^{\alpha+1}} + 1),
\]
that is
\begin{equation}\label{eq22}
\|m\|^{\delta'}_{L^{\alpha +1}} \le C_4(\|\mu\|_{L^{\alpha + 1}} + 1),
\end{equation}
where $\delta' = (\delta(\alpha + 1)-1)/\alpha > 1$. This readily implies
\[
\intTn m^{\alpha + 1} \, dx\le C_5(\xi^{1/\delta'} + 1) \le \xi,
\]
provided that $\xi$ is large enough.


\underline{Case 2}: $\alpha\ge\gamma'/N $. Since \eqref{EI2} is not available anymore, one has to exploit \eqref{EI1}, that is
\begin{equation}\label{eq23}
\|m\|^{\delta}_{L^{\alpha +1}} \le C_6(C_f \|\mu\|^\alpha_{L^{\alpha + 1}} \|m\|_{L^{\alpha+1}} + 1),
\end{equation}
hence
\[
\left(\intTn m^{\alpha + 1} \, dx\right)^\delta \, dx \le C_7\left(C_f \xi^\alpha \intTn m^{\alpha + 1}  \, dx + 1\right).
\]
In this case, one may argue that there exists $\xi$ (large, and depending on $C_7, \alpha, \delta$) such that
\[
\intTn m^{\alpha + 1} \, dx \le \xi,
\]
provided that $C_f$ is small enough. In particular, one  might choose $C_f, \xi$ such that $\xi$ is the smallest root of $\xi^\delta = C_7(C_f \xi^{\alpha+1} + 1)$. 

Note that, in both cases, $\F(\mathcal{K}_\xi) \subseteq \mathcal{K}_\xi$ for all $k$, because $\xi$ does not depend on $k$.

\end{proof}

{\it End of the proof of Lemma \ref{lemma_welldef}}. Since $\F$ is well-defined, continuous, compact and maps $\mathcal{K}_\xi$ into itself, the existence of a fixed point of $\F$, and therefore a solution of \eqref{mkdef}, follows by Schauder theorem.

As for the bounds \eqref{energybound}, we have already obtained that $\lambda_k$ is uniformly bounded from above. Note also that $m_k \in \mathcal{K}_\xi$, that is $\int {m_k}^{\alpha + 1} \, dx \le \xi$, hence \eqref{eq21} implies that
\[
 \intTn |\nabla u_k|^\gamma m_k \, dx \le C_2 + C_H C_f \|m_k\|^{\alpha+1}_{L^{\alpha + 1}} \le C_2 + C_H C_f \xi .
\]
It remains to bound from below $\lambda_k$; one may read \eqref{eq20} as
\[
- \lambda_k \le  C_H + C_f \|(m_{k} \star \psi_k)^\alpha m_k\|_{L^1} \le C_H + C_f \|m_{k}\|^\alpha_{L^{\alpha + 1}} \|m_k\|_{L^{\alpha+1}} \le C_H + C_f \xi,
\]
which proves the assertion.
\end{proof}

The aim is now to pass to the limit as $k \to \infty$, and prove that $(u_k, \lambda_k, m_k)$ converges to a solution of \eqref{MFG}.

\begin{lem}\label{lem_infty} If $\alpha < \gamma'/N$, there exists $C > 0$ such that
\begin{equation}\label{inftybound}
\|m_k\|_{L^\infty} \le C
\end{equation}
for all $k \ge 1$. Otherwise, if $\alpha \ge \gamma'/N$, then \eqref{inftybound} holds under the additional condition that $C_f$ is small enough.
\end{lem}

\begin{proof} By contradiction, let $M_k > 0, x_k \in \Tn$ be such that
\[
0 < M_k := m_k(x_k) = \max_{\Tn} m_k \to \infty, \quad \text{as $k \to \infty$}.
\]
Let us define the following blow-up sequences 
\begin{equation}\label{blowupdef}
v_k(x) := a_k^{\gamma'-2} u(x_k + a_k x), \quad \mu_k(x) := \frac{1}{M_k} m_k(x_k + a_k x), \quad a_k = M_k^{-\alpha/\gamma'},
\end{equation}
for all $x \in \RsetN$. Then, $v_k, \Lambda_k, \mu_k$ solve
\begin{equation}\label{MFGn}
\begin{cases}
- \Delta v_k(x) + H_k(\nabla v_k(x)) +\Lambda_k = W_k(x) - F_k(\mu_{k} \star \hat{\psi}_k(x)))  \\
- \Delta \mu_k(x) -{\rm div}(\nabla H_k(\nabla v_k(x)) \, \mu_k(x)) = 0  & \text{in $\mathbb{T}_k^N$},
\end{cases}
\end{equation}
where $\mathbb{T}_k^N = \{x \in \RsetN : x_k + a_k x \in \Tn \}$ and
\begin{equation}\label{blowup2}
\begin{split}
\bullet \quad & H_k(p) = a_k^{\gamma'}H(a_k^{1-\gamma'}p), \\
\bullet \quad & \Lambda_k = a_k^{\gamma'} \lambda_k, \\
\bullet \quad & W_k(x) = a_k^{\gamma'}V(x_k + a_k x), \\
\bullet \quad & F_k(\mu) = a_k^{\gamma'}f(M_k \mu),
\end{split}
\end{equation}
and $\hat{\psi}_k(x) = a^N_k \psi_k(a_k x)$, and hence it is a radial mollifier. Note that \[a_k \to 0,\] so $H_k$ satisfies \eqref{Hass}, where $C_H$, $\gamma$ are the same as for $H$, and $\Lambda_k \to 0$ by \eqref{energybound}. As for $W_k$, one has that $W_k \to 0$ uniformly on $\mathbb{T}_k^N$, in view of \eqref{Vass1}. Moreover $\mu_{k} \le 1$ on $\mathbb{T}_k^N$, so $\mu_{k} \star \hat{\psi}_k(x) \le 1$, and 
\[
0 \le F_k(\mu_{k} \star \hat{\psi}_k(x))) \le C_f a_k^{\gamma'} (M_k^\alpha + 1) \le 2 C_f
\]
for all $k$, by \eqref{fass}. The couple $(v_k, \Lambda_k)$ solves the HJB equation in \eqref{MFGn}, hence the a-priori estimates of Proposition \ref{hjb_regularity} apply, namely $\|\nabla u_k\|_{L^q(B_2)} \le C_q$, for all $q > 1$. On the other hand, $\mu_k$ is a solution of the Kolmogorov equation in \eqref{MFGn}, and it is uniformly bounded with respect to $k$, so Proposition \ref{ell_regularity} guarantees that $\|\mu_k\|_{W^{1,p}(B_1)} \le C_1$, for some $p > N$ (by choosing $q$ large enough). By standard Sobolev embeddings one may conclude that $\|\mu_k\|_{C^{0,\theta}(B_1)} \le C_2$ for some $\theta > 0$, and finally that
\[
\mu_k(x) \ge C_3 > 0 \quad \text{on $B_r$}
\]
for some positive $r, C_3$ not depending on $k$, as $\mu_k(0) = 1$. Therefore,
\begin{equation}\label{belowLalpha}
\int_{B_r} \mu^{\alpha + 1}_k \, dx \ge C_4 > 0.
\end{equation}

On the other hand, recalling \eqref{energybound} (the assumptions of Lemma \ref{lemma_welldef} are satisfied),
\begin{multline*}
\int_{\mathbb{T}_k^N} \mu_k^{\alpha +1}(x) \, dx = \frac{1}{M_k^{\alpha+1}} \int_{\mathbb{T}_k^N} m_k^{\alpha +1}(x_k + a_k x) \, dx = \frac{1}{a_k^N M_k^{\alpha+1}} \intTn m_k^{\alpha +1}(x) \, dx\\ \le \frac{C_5}{M_k^{\alpha+1 - \alpha N/\gamma'}} \rightarrow 0 \quad \text{as $k \rightarrow \infty$},
\end{multline*}
because $M_k \rightarrow \infty$, and $\alpha+1 - \alpha N/\gamma' > 0$ if and only if \eqref{critical_alpha} holds, but this contradicts \eqref{belowLalpha}.

\end{proof}

{\it End of the proof of Theorem \ref{thm_existence}.} Once the crucial estimates \eqref{energybound} and \eqref{inftybound} are established, it is sufficient to pass to the limit as $k \to \infty$. Note that uniform $L^\infty$ bounds on $m_k$ imply uniform $W^{1, p}$ bounds on $u_k$ (by Proposition \ref{hjb_regularity}), that in turn provide $W^{1, p}$ and $C^{0, \alpha}$ bounds for $m_k$ (apply Proposition \ref{ell_regularity} and Sobolev embeddings). By Schauder regularity, $u_k$ is bounded in $C^{2,\alpha}(\Tn)$, so, up to subsequences,
\[
u_k \to u \quad \text{in $C^2(\Tn)$}, \quad m_k \rightharpoonup m \quad \text{in $W^{1,p}(\Tn)$}, \quad \lambda_k \to \lambda,
\]
and by passing to the limit in \eqref{mkdef}, one has that $(u, \lambda, m)$ solves \eqref{MFG}.

\end{proof}

\begin{rem}\label{rem:var} In order to better understand the features of system \eqref{MFG}, it can be useful to consider its \emph{variational} formulation (see \cite{LasryLions}), in the particular case of $H,f$ satisfying \eqref{Hf_partic} (we also choose $V \equiv 0$ for simplicity): solutions of \eqref{MFG} are associated, at least formally, to critical points of the functional
\[
\J(A, m) = \frac{1}{\gamma'} \intTn |A|^{\gamma'} m - \frac{C_f}{\beta}\intTn m^{\beta},
\]
where $\beta = \alpha + 1$, subject to the constraint that the vector field $A : \Tn \to \RsetN$ and the distribution $m$ are coupled via the Kolmogorov equation
\begin{equation}
- \Delta m + {\rm div}(A \, m) = 0, \quad \intTn m \, dx = 1, \, \quad m > 0.
\end{equation}
The functional $\J$ can be regarded as the \emph{energy} of the system; it is convex in the defocusing case (where $-m^\beta$ is replaced by $m^\beta$), but it is not convex when the coupling is focusing. One may ask whether or not $\J$ is at least bounded from below (or above), if $A,m$ are suitably chosen. At this level, the key insight comes from Proposition \ref{prop_energy}, which states that
\begin{equation}\label{SI_vs_KIN}
\left(\intTn m^\beta  \, dx\right)^\delta \le C\left(\intTn |A|^{\gamma'} m  \, dx  + 1\right)
\end{equation}
for some $C > 0$, $\delta > 1$, at least if $\beta < 1 + \gamma'/N$, that is $\alpha < \gamma'/N$. With this inequality in mind, that deeply relies on the Gagliardo-Nirenberg inequality and $L^1$-constraint of $m$, one may argue that $\J$ is bounded from below among $A,m$ satisfying $\int |A|^{\gamma'} m < \infty$, and direct minimization of $\J$ makes sense. In other words, if the growth of $f$ at infinity is sufficiently mild, the ``kinetic'' part of $\J$, represented by $\intTn |A|^{\gamma'} m$, is stronger than the ``self-interaction'' term $\int m^\beta$. 

Proposition \ref{prop_energy} states that while \eqref{SI_vs_KIN} is not available anymore if $\beta \ge 1 + \gamma'/N$, boundedness of $\int |A|^{\gamma'} m$ still guarantees compactness of $m$ in $L^\beta(\Tn)$ if $\beta < 1 + \gamma'/(N-\gamma')$, that is when $\alpha < {\gamma'}/({N-\gamma'})$ (see Remark \ref{rem_energy}). Hence, even if direct minimization of $\J$ might not be possible, the problem should posses enough compactness, and a critical point of $\J$ might exist.

All these considerations are inspired by some classical results within the theory of \emph{focusing} nonlinear Schr\"odinger equations (NLS). The link between the MFG system \eqref{MFG} and NLS equations can be made precise in the quadratic case $\gamma = \gamma' = 2$, by the observation that if $H$ satisfies \eqref{Hf_partic}, then $\varphi^2 := m = e^{-u} / \int e^{-u}$ trivializes the Kolmogorov equation in \eqref{MFG}, and \eqref{MFG} reduces to the following Hartree equation
\[
-2 \Delta \varphi = \lambda \varphi + C_f \varphi^{2\alpha + 1}, \quad \intTn \varphi^2\, dx = 1, \, \quad \varphi > 0,
\]
which is associated to the functional
\[
\widetilde{\mathcal{E}}(\varphi) = \intTn |\nabla \varphi|^2 \, dx - \frac{C_f}{2\beta} \intTn \varphi^{2 \beta} \, dx, \quad \text{subject to $\intTn \varphi^2 = 1$}.
\]
This minimization problem is well-understood in the so-called $L^2$-subcritical regime, namely when $2 \beta < 2 + 4/N$: in this case, $\widetilde{\mathcal{E}}(\varphi)$ is bounded from below in view of the Gagliardo-Nirenberg inequality and the $L^2$ constraint, and a minimizer can be found via the direct method of calculus of variations (see, for example, \cite{CazLions, Caz} for a description in $\RsetN$). Note that $2 \beta < 2 + 4/N$ is equivalent to $\alpha < 2/N$, and $2/N$ is precisely the first exponent $\alpha_1 = \gamma'/N$ that appears in our analysis: we may regard $\alpha_1$ as the \emph{mass-critical} exponent.

In the $L^2$-supercritical regime, $\widetilde{\mathcal{E}}(\varphi)$ is not bounded from below; if $2 \beta < 2N/(N-2)$, one expects the existence of a critical point under the additional condition that the $L^2$ constraint is small, or equivalently that $C_f$ is small (see \cite{NTV}, where the setting of a bounded domain $\Omega$ and homogeneous Dirichlet boundary data is considered). The condition $2 \beta < 2N/(N-2)$ can be read as $\alpha < 2/(N-2)$, and $2/(N-2)$ is $\alpha_2 = \gamma'/(N-\gamma')$ in the quadratic case $\gamma' = 2$; in the NLS jargon, $\alpha_2$ is the \emph{energy-critical} exponent.

If $2\beta \ge 2N/(N-2)$, the lack of Sobolev embeddings and compactness is a serious issue in variational methods and the classical Pohozaev identity implies non-existence of solutions in many situations.

Back to the general MFG system \eqref{MFG}, we prove that it exhibits the same behaviour as NLS: the two critical values $\alpha_1$ and $\alpha_2$ are such that if $\alpha \in (0, \alpha_1)$, then \eqref{MFG} has a solution, if $\alpha \in [\alpha_1, \alpha_2)$ a solution exists if the additional condition that $C_f$ be small is imposed, and if $\alpha \in [\alpha_2, \infty)$ solutions may not exist. In view of this analogy with focusing NLS, we have decided to call focusing the MFG systems with decreasing (and unbounded) coupling.
\end{rem}

\section{A Pohozaev identity and non-existence of solutions}\label{s:nonex}


This section in devoted to the proof of non-existence of solutions in the ``supercritical'' regime $\alpha \ge \gamma'/(N-\gamma')$. We start with an integral identity, which is produced by implementing in the MFG setting a celebrated idea of Pohozaev (see \cite{Pohoz}).

\begin{prop}\label{pohozaev} Let $\Omega \subset \Rset^N$ be a bounded and smooth domain, and $g, G \in C^1(\overline{\Omega} \times \Rset)$ be such that
\[
\partial_m G(x,m) = g(x,m) \quad \forall x,m.
\]
Suppose that $(u, \lambda, m) \in C^2(\overline{\Omega})\times \Rset \times W^{1,1}(\Omega)$ is a solution of
\begin{equation}\label{MFGg}
\begin{cases}
- \Delta u(x) + \frac{1}{\gamma} {|\nabla u(x)|^\gamma} = g(x,m(x))  \\
- \Delta m(x) -{\rm div}(|\nabla u(x)|^{\gamma-2}\nabla u(x) \, m(x)) = 0  & \text{in $\Omega$}.  \\
\end{cases}
\end{equation}
Then, the following equality holds
\begin{equation}
N \int_\Omega G(x,m) \, dx+ \int_\Omega \nabla_x G \cdot x \, dx + \left(1- \frac{N}{\gamma}\right) \int_\Omega |\nabla u|^\gamma m + (2-N) \int_\Omega \nabla u \cdot \nabla m \, dx = I_{\partial \Omega},
\end{equation}
where
\begin{multline*}
I_{\partial \Omega} = \int_{\partial \Omega} \left( G - \nabla u \cdot \nabla m - \frac{1}{\gamma}|\nabla u|^\gamma m \right) \,  x \cdot \nu \, d \sigma \\
+ \int_{\partial \Omega}  (\nabla u \cdot \nu)(\nabla m \cdot x) + (\nabla m \cdot \nu)(\nabla u \cdot x) + |\nabla u|^{\gamma-2}(\nabla u \cdot \nu)( \nabla u \cdot x) m \, d \sigma
\end{multline*}
\end{prop}

\begin{proof} As we already mentioned in the introduction, if $1 < \gamma < 2$ the Kolmogorov equation is solved in the weak sense (while $m$ is twice differentiable everywhere if $\gamma \ge 2$, and all the following computations are justified). In this case, it is sufficient to approximate the problem by replacing $H(p) = |p|^\gamma/\gamma$ by $H_\epsilon(p) = (\epsilon + |p|^2)^{\gamma/2}/\gamma$: in this way, $m = m_\epsilon$ is a-posteriori a classical solution of the Kolmogorov equation, and the statement of the theorem follows by letting $\epsilon \to 0$.

The following equality is obtained by integrating by parts and exchanging the order of derivation\footnote{For the sake of brevity, the Einstein summation convention on repeated indices is used here.}:
\begin{multline}\label{eq1}
\int_\Omega \nabla u \cdot \nabla (\nabla m \cdot x) \, dx = \int_\Omega u_{x_i} m_{x_i x_j} x_j \, dx +  \int_\Omega \nabla u \cdot \nabla m \, dx = \\
- \int_\Omega m_{x_i} u_{x_i x_j} x_j \, dx +  (1-N) \int_\Omega \nabla u \cdot \nabla m \, dx + \int_{\partial \Omega} \nabla u \cdot \nabla m \,  x \cdot \nu d \sigma = \\
- \int_\Omega \nabla m \cdot \nabla (\nabla u \cdot x) \, dx +  (2-N) \int_\Omega \nabla u \cdot \nabla m \, dx + \int_{\partial \Omega} \nabla u \cdot \nabla m \,  x \cdot \nu d \sigma.
\end{multline}
One sees that
\begin{equation}\label{eq2}
\frac{1}{\gamma}\nabla(|\nabla u|^\gamma)\cdot x = |\nabla u|^{\gamma-2} u_{x_i} u_{x_i x_j} x_j = |\nabla u|^{\gamma-2} \nabla u \cdot \nabla(\nabla u \cdot x) - |\nabla u|^\gamma,
\end{equation}
hence
\begin{multline}\label{eq3}
- \frac{1}{\gamma} \int_\Omega \nabla(|\nabla u|^\gamma) \cdot x \, m \, dx = \int_\Omega \nabla m \cdot \nabla (\nabla u \cdot x) \, dx + \int_\Omega |\nabla u|^\gamma m \, dx \\
-\int_{\partial \Omega} (\nabla m \cdot \nu)(\nabla u \cdot x) \, d \sigma - \int_{\partial \Omega} |\nabla u|^{\gamma-2}(\nabla u \cdot \nu)( \nabla u \cdot x) m \, d \sigma,
\end{multline}
by exploiting  the second equation of \eqref{MFGg} multiplied by $\nabla u \cdot x$. Therefore, by multiplying \eqref{MFGg} by $\nabla m \cdot x$ and integrating by parts we obtain
\begin{multline*}
\int_\Omega g \, \nabla m \cdot x \, dx = - \int_\Omega \Delta u (\nabla m \cdot x) \, dx +\frac{1}{\gamma} \int_\Omega  |\nabla u|^\gamma (\nabla m \cdot x) \, dx =  -\int_{\partial \Omega} (\nabla u \cdot \nu)(\nabla m \cdot x) \, d \sigma + \\
\int_\Omega \nabla u \cdot \nabla (\nabla m \cdot x) \, dx - \frac{1}{\gamma} \int_\Omega {\rm div}(|\nabla u|^\gamma x) m \, dx + \frac{1}{\gamma} \int_{\partial \Omega} |\nabla u|^\gamma m \,  x \cdot \nu \, d \sigma, \\
\end{multline*}
that is, plugging in \eqref{eq1} and \eqref{eq3},
\begin{multline}\label{eq4}
\int_\Omega g \, \nabla m \cdot x \, dx = 
- \int_\Omega \nabla m \cdot \nabla (\nabla u \cdot x) \, dx - \frac{N}{\gamma} \int_\Omega |\nabla u|^\gamma m \, dx  - \frac{1}{\gamma} \int_\Omega \nabla(|\nabla u|^\gamma) \cdot x \, m \, dx  \\ + (2-N) \int_\Omega \nabla u \cdot \nabla m \, dx + \int_{\partial \Omega} \left( \nabla u \cdot \nabla m + \frac{1}{\gamma}|\nabla u|^\gamma m\right) \,  x \cdot \nu  - (\nabla u \cdot \nu)(\nabla m \cdot x) \, d \sigma, = \\
\left(1 - \frac{N}{\gamma}\right) \int_\Omega |\nabla u|^\gamma m \, dx  + (2-N) \int_\Omega \nabla u \cdot \nabla m \, dx +
\int_{\partial \Omega} \left( \nabla u \cdot \nabla m + \frac{1}{\gamma}|\nabla u|^\gamma m\right) \,  x \cdot \nu\, d \sigma \\ - \int_{\partial \Omega}  (\nabla u \cdot \nu)(\nabla m \cdot x) + (\nabla m \cdot \nu)(\nabla u \cdot x) + |\nabla u|^{\gamma-2}(\nabla u \cdot \nu)( \nabla u \cdot x) m \, d \sigma.
\end{multline}
Finally, integrating again by parts,
\begin{equation}
-N \int_\Omega G \, dx = - \int_\Omega G\,  {\rm div}(x) \, dx =  \int_\Omega \nabla_x G \cdot x \, dx +  \int_\Omega g \, \nabla m \cdot x \, dx - \int_{\partial \Omega} G \,  x \cdot \nu \, d \sigma
\end{equation}
which gives the stated equality in view of \eqref{eq4}.
\end{proof}

The following basic lemma will be useful to control the behavior of $u,m,\nabla u, \nabla m$ at infinity. We provide the proof for the convenience of the reader (following the lines of \cite{BerestyckiLions}, where the Pohozaev identity is used in $\RsetN$ in the setting of NLS equations).

\begin{lem}\label{decay} Let $h \in L^1(\RsetN)$. Then, there exists a sequence $R_n \to \infty$ such that
\[
R_n \int_{\partial B_{R_n}} |h(x)| \, dx \rightarrow 0 \quad \text{as $n \to \infty$}.
\]
\end{lem} 

\begin{proof} By the co-area formula, it holds true that
\begin{equation}\label{eq30}
\intRn |h(x)|\,dx = \int_0^\infty \int_{\partial B_{R}} |h(x)| \, dx \, dR < \infty.
\end{equation}
If, by contradiction, 
\[
\liminf_{R \rightarrow \infty} R \int_{\partial B_{R}} |h(x)| \, dx = \alpha > 0,
\]
then
\[
R \mapsto \int_{\partial B_{R}} |h(x)| \, dx
\]
would not be in $L^1(0, +\infty)$, that is not compatible with \eqref{eq30}.
\end{proof}

While the negativity of the ergodic constant $\lambda$ is a direct consequence of the comparison principle when the state space $\Omega$ is the flat torus, if the problem is set on $\RsetN$ the following lemma is required.

\begin{lem}\label{lambdanegative}
Suppose that $(u, \lambda) \in C^2(\RsetN) \times \Rset$ is a solution of the HJB equation in \eqref{MFG}, where $H$ is of the form \eqref{Hf_partic}, $V \equiv 0$, and $u$ satisfies \eqref{uinfty}. Then,
\[
\lambda \le 0.
\]
\end{lem}

\begin{proof} Let $\mu_\epsilon(x) := \epsilon^{N/2} (2 \pi)^{-N/2} e^{-\epsilon|x|^2/2}$ for all $x \in \RsetN$, $\epsilon > 0$. Note that $\intRn \mu_\epsilon = 1$ for all $\epsilon$.

Since $|p|^\gamma/\gamma$ is the Legendre transform of $|p|^{\gamma'}/\gamma'$, one observes that
\[
- \Delta u(x) + \nabla u(x) \cdot (\epsilon x) - \frac{1}{\gamma'}|\epsilon x|^{\gamma'} + \lambda \le - \Delta u(x) + \sup_{\alpha \in \RsetN}\left\{\nabla u(x) \cdot \alpha - \frac{1}{\gamma'}|\alpha|^{\gamma'} \right\} + \lambda = -f(m(x)),
\]
by choosing $\alpha = \alpha(x) = \epsilon x$. Multiplying this equality by $\mu_\epsilon$ on $B_R$, $R > 0$ and integrating by parts gives
\begin{multline*}
\int_{B_R} \nabla u \cdot \nabla \mu_\epsilon \, dx - \int_{\partial B_R} \mu_\epsilon  \nabla u \cdot \nu \, dx + \int_{B_R} \nabla u \cdot (\epsilon x) \mu_\epsilon \, dx + \lambda \int_{B_R} \mu_\epsilon \, dx \le\\ \frac{1}{\gamma'} \int_{B_R} |\epsilon x|^{\gamma'} \mu_\epsilon \, dx - \int_{B_R} f(m(x)) \mu_\epsilon \, dx.
\end{multline*}
Recall that $f \ge 0$ and $\nabla \mu_\epsilon = -\epsilon x \mu_\epsilon$, so one obtains
\begin{equation}\label{eq31}
 \lambda \int_{B_R} \mu_\epsilon \, dx \le \frac{1}{\gamma'} \int_{B_R} |\epsilon x|^{\gamma'} \mu_\epsilon \, dx + \int_{\partial B_R} \mu_\epsilon  \nabla u \cdot \nu \, dx.
\end{equation}
The first integral of the right-hand side of \eqref{eq31} can be made arbitrarily small as $\epsilon \to 0$, that is $\int_{B_R} |\epsilon x|^{\gamma'} \mu_\epsilon  \le C \epsilon^{\gamma'/2}$, where $C > 0$ does not depend on $\epsilon$. Moreover, for any $\epsilon > 0$ fixed,
\[
\left| \int_{\partial B_R} \mu_\epsilon  \nabla u \cdot \nu \, dx \right| \le C \epsilon^{N/2} |\partial B_R| \| \nabla u \|_{L^\infty(\partial B_R)} e^{-\epsilon|R|^2/2} \to 0 \quad \text{as $R \to \infty$},
\]
because, $|\partial B_R|, \| \nabla u \|_{L^\infty(\partial B_R)} $ grow polynomially in $R$ as $R \to \infty$. Therefore, letting $R \to \infty$ in \eqref{eq31} provides
\[
\lambda \le C \epsilon^{\gamma'/2},
\]
and the stated assertion follows by letting $\epsilon \to 0$.
\end{proof}

We are now ready to prove the non-existence theorem.

\begin{proof}[Proof of Theorem \ref{thm_nonex}]
Suppose, by contradiction, that $(u, \lambda, m) \in C^2(\RsetN) \times \Rset \times W^{1,1}(\RsetN)$ is a solution of \eqref{MFG} satisfying
\eqref{uinfty}-\eqref{finiteness}. Let us define $g(x,m) = - \lambda- f(m), G(x,m) = - \lambda m - F(m)$; of course, $\partial_m G(x,m) = g(x,m)$ for all $x,m$. We argue that, by Proposition \ref{pohozaev}, it holds true that
\begin{equation}\label{eq43}
N \intRn G(x,m) \, dx + \left(1- \frac{N}{\gamma}\right) \intRn |\nabla u|^\gamma m + (2-N) \intRn \nabla u \cdot \nabla m \, dx = 0.
\end{equation}
In particular, the proposition applies if we choose $\Omega = B_R$, $R > 0$. In order to obtain \eqref{eq43}, it suffices to observe that
\[
|I_{\partial B_{R_n}}| \le R_n \int_{\partial {B_{R_n}}} ( |G| + 3 |\nabla u| | \nabla m| + 2|\nabla u|^\gamma m )  \, d \sigma \to 0,
\]
for some sequence $R_n \to \infty$, because $G, |\nabla u| | \nabla m|, |\nabla u|^\gamma m \in L^1(\RsetN)$ (use Lemma \ref{decay}).

It also holds true that
\begin{align}
\intRn \nabla u \cdot \nabla m \, dx & = - \intRn |\nabla u|^{\gamma} m \, dx \label{eq41} \\ 
\intRn \nabla u \cdot \nabla m \, dx & = \intRn g(x,m(x))m(x) \, dx - \frac{1}{\gamma} \intRn |\nabla u|^{\gamma} m \label{eq42} \, dx.
\end{align}
Indeed, set
\begin{equation}\label{omegas}
\text{$\Omega_s := \{x \in \RsetN : u(x) \le s\}\quad$ and $\quad v_s(x) := u(x) - s$ for all $s > 0, x \in \RsetN$.}
\end{equation}
Without loss of generality, we may suppose that $u(0) = 0$. In this situation, every $\Omega_s$ is non-empty and bounded, and $\cup_s \Omega_s = \RsetN$, by the fact that $u \to \infty$ as $|x| \to \infty$. If one multiplies the Kolmogorov equation in \eqref{MFG} by $v_s$, an integration by parts yields
\[
\int_{\Omega_s} \nabla v_s \cdot \nabla m \, dx = - \int_{\Omega_s} |\nabla u|^{\gamma-2} \nabla u \cdot \nabla v_s m \, dx.
\]
Note that $\nabla u = \nabla v_s$ for all $s$, so \eqref{eq41} follows by letting $s \to \infty$. As for \eqref{eq42}, it suffices to multiply the HJB equation in \eqref{MFG} by $m$ and integrate by parts on a generic ball $B_R$, that is
\[
\int_{B_R} \nabla u \cdot \nabla m\, dx - \int_{\partial B_R} m \nabla u \cdot \nu \, dx + \frac{1}{\gamma}\int_{B_R} |\nabla u|^\gamma \, dx  = \int_{B_R} g(m) m \, dx.
\]
A straightforward application of the Holder inequality leads to
\begin{equation}\label{eq40}
\intRn |\nabla u| m \, dx \le \left(\intRn |\nabla u|^\gamma m \, dx \right)^{1/\gamma} \left(\intRn m \, dx \right)^{1/\gamma'} < \infty,
\end{equation}
so 
\[
\int_{\partial B_{R_n}} m \nabla u \cdot \nu \, dx \to 0 \quad \text{as $R_n \to \infty$},
\]
along an appropriate sequence $R_n \to \infty$, in view of \eqref{eq40} and Lemma \ref{decay}, and \eqref{eq42} follows.

Equations \eqref{eq41} and \eqref{eq42} are now plugged into \eqref{eq43} to get
\[
N \intRn G(m) \, dx - (N- \gamma') \intRn g(m)m \, dx = 0,
\]
that is
\[
\intRn [ (N-\gamma') f(m) m - N F(m) ] \, dx = \lambda \gamma' \le 0
\]
by $\intRn = 1$ and Lemma \ref{lambdanegative}. Therefore, recalling \eqref{fass2}, a contradiction is reached.
\end{proof}

\begin{rem}\label{rem:infinity} We end this section with some remarks about conditions \eqref{uinfty} and \eqref{finiteness}. As for $u \to \infty$ as $x \to \infty$, it is a quite natural ``boundary'' condition for ergodic HJB equations on the whole space. Observe that the optimal control $-\nabla H(\nabla u)$ should give rise to an ergodic process, and this happens heuristically if $-\nabla H(\nabla u) \cdot x < 0$ for $x$ large, that is $\nabla u \cdot x > 0$ when $H$ is the model Hamiltonian \eqref{Hf_partic}. We refer to \cite{CirantErgodic} (and references therein) for more information about ergodic problems on the whole space. The polynomial control at infinity of $\nabla u$ is more technical than substantial, and it is usually true under mild assumptions: for example, if $\|m\|_{W^{1,\infty}(\RsetN)} < \infty$, one may invoke standard Bernstein estimates to obtain a uniform control of $\nabla u$ on the whole space (see, for example, \cite{lasry1989nonlinear}).

As we have seen in Remark \ref{rem:var}, boundedness of $|\nabla u|^\gamma m$, $m^{\alpha+1}$ in $L^1(\RsetN)$ is a natural requirement for the well-posedness of \eqref{MFG}. If one just looks at the Kolmogorov equation, integrability of the vector field $-\nabla H(\nabla u)$ with respect to $m$ is recommended to ensure some minimal regularity of $m$ and  uniqueness of the invariant distribution itself (see \cite{Metafune}).

Integrability of $|\nabla u| |\nabla m|$ is another technical requirement; it is a consequence of the other conditions in the quadratic case, and we expect the same also in more general situations: the decay at infinity of $m$ is usually exponential, while the growth of $|\nabla u|$ is polynomial (see \cite{CirantErgodic}).

To conclude, it is not the aim of this work to pursue non-existence results under minimal conditions on $u,m$ at infinity; our focus here is concentration of $m$ caused by criticality phenomena, rather than a careful analysis of the behavior of the problem at infinity.
\end{rem}

\section{Uniqueness in very special cases}\label{s:unique}

The aim of this section is to obtain uniqueness results for the system \eqref{MFG}, in the case
\begin{equation}\label{polyHf}
\Omega = \RsetN, \qquad V \equiv 0, \qquad H(p) = \frac{1}{2} |p|^2, \qquad f(m) = m^\alpha, \quad \alpha > 0.
\end{equation}
Even though in what follows we might multiply $H, f$ by some positive constants and obtain analogous results, we stress that the assumption \eqref{polyHf} is still quite rigid and selects a very special class of focusing MFG systems. The quadratic form of $H$ enables us to exploit the standard Hopf-Cole transformation (see \cite{LasryLions}), in the following way.

\begin{lem}\label{hopfcole} Let $(u, \lambda, m)$ be a classical solution of \eqref{MFG}-\eqref{uinfty}, such that
\begin{equation}\label{finite_energy}
\intRn |\nabla u|^2 m < \infty.
\end{equation}
Then, $\varphi := e^{-u/2}$ satisfies
\begin{equation}\label{semilin}
\begin{cases}
2 \Delta \varphi + \lambda \varphi + \varphi^{2\alpha +1} = 0 & \text{in $\RsetN$} \\
\varphi(x) \rightarrow 0 \quad\text{as $x \rightarrow \infty$} \\
\int_\RsetN \varphi^2 \, dx = 1, \, \varphi > 0,
\end{cases}
\end{equation}
and $m = \varphi^2$.
\end{lem}

\begin{proof} By exploiting the HJB equation in \eqref{MFG}, an easy computation shows that $\varphi = e^{-u/2}$ solves
\begin{equation}\label{eq51}
2 \Delta \varphi + \lambda \varphi + m^{\alpha}\varphi = 0 \quad \text{in $\RsetN$}.
\end{equation}
Since $u \to +\infty$ as $x \to \infty$, $\varphi$ vanishes at infinity. Note that $m$ and $\mu := e^{-u}$ are both solution of the Kolmogorov equation
\begin{equation}\label{kolmoB}
-\Delta m +{\rm div}(A \, m) = 0 \quad \text{on $\RsetN$,}
\end{equation} where $A = - \nabla u$. Applying uniqueness results of probability solutions for \eqref{kolmoB} (see for example \cite{BRS11} and references therein), since $\intRn |A|^2 m < \infty$, one has $m \equiv \mu = \varphi^2$. Substituting the last equality into \eqref{eq51} concludes the proof.
\end{proof}

\begin{prop} Suppose that $0 < \alpha< 2/(N-2)$ and $\alpha \neq 2/N$. Then, there exists a unique classical solution $(u, \lambda, m)$ of \eqref{MFG} such that
\[
\intRn |\nabla u|^2 m < \infty, \qquad \text{and} \quad \intRn x \, m \, dx = 0.
\]
Moreover, $m$ and $u$ are radially symmetric around the origin.
\end{prop}

\begin{proof} Suppose that $(u_i, \lambda_i, m_i)$, $i = 1,2$ are classical solutions of \eqref{MFG}. By Lemma \ref{hopfcole}, $\lambda_i$ and $\varphi_i = e^{-u_i/2} = \sqrt{m_i}$ are classical solutions of the semilinear elliptic equation \eqref{semilin}. Firstly, $\int |\nabla u_i|^2 m_i < \infty$ reads
\[
\intRn |\nabla \varphi_i|^2 < \infty.
\]
Moreover, $\lambda_i < 0$, because $2 \Delta \varphi + \lambda \varphi + \varphi^{2\alpha +1} = 0$ has no non-trivial solutions in $\RsetN$ if $\lambda \ge 0$.

Set $\psi_i(x) = |\lambda_i|^{-1/(2\alpha)} \varphi_i (x/|\lambda_i|^{1/2})$. Then, $\psi_1$ and $\psi_2$ are positive solutions of
\begin{equation}\label{normalized_semilin}
2 \Delta \psi - \psi + \psi^{2\alpha +1} = 0 \quad \text{in $\RsetN$},
\end{equation}
vanishing at infinity. Standard results on symmetry of solutions of semilinear equations imply that $\psi_i$ are radially symmetric, i.e. there exist $x_i \in \RsetN$, $i = 1,2$ such that $\psi_i(x-x_i)$ are radially symmetric around the origin (see, for example, \cite{LiNi}). Note that $\int x \, \varphi^2_i \, dx = 0$ is required, so $x_1 = x_2 = 0$.

As uniqueness holds for radial solutions of \eqref{normalized_semilin} (see \cite{Kwong}), one has
\[
\psi_1(x) = \psi_2(x) \quad \text{for all $x \in \RsetN$}.
\]
Since $\int \varphi_i^2 = 1$, one obtains
\[
|\lambda_1|^{\frac{N\alpha - 2}{2 \alpha}} = \intRn \psi_1^2 \, dx = \intRn \psi_2^2 \, dx = |\lambda_2|^{\frac{N\alpha - 2}{2 \alpha}},
\]
that implies $\lambda_1 = \lambda_2$ whenever $\alpha \neq 2/N$, and consequently $\varphi_1 \equiv \varphi_2$. By the equality $\varphi_i = e^{-u_i/2} = \sqrt{m_i}$ one obtains the assertion on uniqueness. As for existence, it follows by standard existence arguments for \eqref{normalized_semilin}.
\end{proof}

\begin{rem} The previous argument uses extensively the special form of $H, f$ and the scaling properties of \eqref{semilin} to reduce the problem of uniqueness of a solution $(\varphi, \lambda)$ to the uniqueness of $\psi$ solving \eqref{normalized_semilin}. Being the equation homogeneous in the $x$-variable, one expects uniqueness ``up-to-translations'', that is, agents have no preferences about the concentration point of $m$ in the state space.

The situation drastically changes if some potential function is added into the problem, namely
\[
V(x) \not\equiv 0.
\]
This case is much more delicate and tricky to treat from the point of view of uniqueness, as it is not possible anymore to remove the unknown $\lambda$ by rescaling. Heuristically, a radially increasing potential $V$ should localize the problem around the origin and boil down the possibility of having multiple ``concentration points''. It is still possible to transform \eqref{MFG} into a semilinear equation (with potential), but uniqueness of a solution $(\varphi, \lambda)$ with $L^2$-constraint is in general an open problem. As far as we know, a complete description of the ground states (positive solutions vanishing as $x \rightarrow \infty$) of 
\begin{equation}\label{semilinV}
2 \Delta \varphi + \lambda \varphi - V(x) \varphi + \varphi^{2\alpha +1} = 0 \quad \text{in $\RsetN$},
\end{equation}
is available only in space dimension $N=1$, with bounded $V$ and in the $L^2$-subcritical case $\alpha < 2/N =1$ (see \cite{mcleod2003}).

We mention that the uniqueness of a couple $(\varphi, \lambda)$ solving \eqref{semilinV} such that $\int \varphi^2 = 1$ is strictly related to the problem of \emph{orbital stability} of ground states of \eqref{semilinV}, that is the long-time stability of standing waves of the associated nonlinear Schr\"odinger equation.

\end{rem}

\begin{rem} If $H$ is non-quadratic, but of the form $H(p) = |p|^\gamma$, $\gamma > 1$, it is possible to transform \emph{radial} solutions of \eqref{MFG} into solutions of a quasi-linear equation involving the $\Delta_{\gamma'}$ operator (see \cite{MR3377677}). A rescaling argument and uniqueness results for this kind of quasilinear equations (see, for example, \cite{pucci1998uniqueness}) then applies as in the quadratic case, leading to uniqueness of radial finite-energy solutions of the MFG system. However, it is not known in general if any $u,m$ solving \eqref{MFG} with non-quadratic Hamiltonian has radial symmetry.
\end{rem}

\begin{rem} We observe that if condition \eqref{finite_energy} does not hold, one may not expect in general uniqueness of solutions for the Kolmogorov equation on $\RsetN$ (see, for example, counterexamples in \cite{BRS11}).
\end{rem}

\appendix

\section{Existence of solutions in the defocusing case}\label{a:defoc}

In this final appendix we will prove the existence theorem for stationary defocusing MFG systems. We will follow the lines of the proof of Theorem \ref{thm_existence}, with emphasis on what needs an adaption with respect to the focusing case.

\begin{proof}[Proof of Theorem \ref{thm_existence_def}] Let $k$ be a positive integer, $\psi$ be a radial mollifier and $\psi_k := k^N \psi(kx)$. Let $m_k$ be the defined by
\begin{equation}\label{mkdef2}
\begin{cases}
- \Delta u_k + H(\nabla u_k) + \lambda_k = V(x) + f(m_{k} \star \psi_k) \star \psi_k \\
- \Delta m_k -{\rm div}(\nabla H(\nabla u_k) \, m_k) = 0  & \text{in $\Tn$},  \\
\intTn m \, dx = 1, \, m > 0.
\end{cases}
\end{equation}
A solution $(u_k, \lambda_k, m_k) \in C^2(\Tn) \times \Rset \times W^{1,p}(\Tn)$ exists for all $k \ge 1$, as $f(\cdot \star \psi_k) \star \psi_k$ is a smoothing operator (see, for example, \cite{LasryLions}, or follow the lines of the proof of Lemma \ref{lemma_welldef}).

\underline{Step 1}. There exists $C > 0$ such that
\begin{equation}\label{energybound2}
|\lambda_k| \le C, \qquad \intTn |\nabla u_k|^{\gamma} m_k \, dx \le C.
\end{equation}
Such estimates are standard in the MFG literature; $\lambda_k$ is positive by the Maximum Principle, and the bounds from above are obtained by combining
\begin{multline*}
\intTn f(m_{k} \star \psi_k) \star \psi_k \, m_k \, dx \le \\ C_H^{-1} \intTn |\nabla u_k|^\gamma m_k \, dx + \intTn V \, m \, dx + \intTn f(m_{k} \star \psi_k) \star \psi_k \, m_k \, dx \le \lambda_k +  C_H,
\end{multline*}
which comes from multiplying the HJB equation in \eqref{mkdef2} by $m_k$ and the Kolmogorov equation by $u_k$, and
\[
-C_H + \lambda_k \le \intTn H(\nabla u_k) \, m_k \, dx + \lambda_k = \intTn V \, dx + \intTn f(m_{k} \star \psi_k) \star \psi_k \, dx,
\]
that is the HJB equation in \eqref{mkdef2} being integrated on $\Tn$.

 Then, Corollary \ref{cor_energy} applies and
\begin{equation}\label{int_bound}
\intTn m^{\alpha + 1}_{k} \, dx \le C
\end{equation}
for some constant $C$ not depending on $k$ (since $\alpha$ satisfies \eqref{critical_alpha}).

\underline{Step 2}. As the integral bound \eqref{int_bound} is established, one may set up the same blow-up procedure as in the proof of Lemma \ref{lem_infty}, in order to obtain $L^\infty$ bounds. In particular, suppose by contradiction that $M_k > 0, x_k \in \Tn$ are such that $M_k := m_k(x_k) = \max_{\Tn} m_k \to \infty$. Then, $v_k(x), \mu_k(x)$ defined as in \eqref{blowupdef} solve
\begin{equation}\label{MFGnn}
\begin{cases}
- \Delta v_k(x) + H_k(\nabla v_k(x)) +\Lambda_k = W_k(x) + F_k(\mu_{k} \star \hat{\psi}_k) \star \hat{\psi}_k (x)  \\
- \Delta \mu_k(x) -{\rm div}(\nabla H_k(\nabla v_k(x)) \, \mu_k(x)) = 0  & \text{in $\mathbb{T}_k^N$},
\end{cases}
\end{equation}
where $\mathbb{T}_k^N = \{x \in \RsetN : x_k + a_k x \in \Tn \}$ and $H_k(p), \Lambda_k , W_k(x), F_k(\mu)$ are as in \eqref{blowup2}. Local a-priori estimates on $\nabla v_k$ and $\mu_k$ provide $\int_{B_r} \mu^{\alpha + 1}_k \, dx \ge c > 0$ (see \eqref{belowLalpha}), which contradicts $\int_{\mathbb{T}_k^N} \mu_k^{\alpha +1}(x) \, dx \le({M_k^{\alpha+1 - \alpha N/\gamma'}} )^{-1} \intTn m_k^{\alpha +1} \, dx \to 0$ as $k \rightarrow \infty$. Hence,
\[
\|m_k\|_{L^\infty} \le C.
\]

\underline{Step 3}. It is now possible to pass to the limit as $k \rightarrow \infty$, and obtain a solution $(u, \lambda, m)$ of \eqref{MFGde} by reasoning as in the end of the proof of Theorem \ref{thm_existence}.
\end{proof}

\small

\bibliographystyle{plain}
\bibliography{bib_focusing}

\medskip
\begin{flushright}
\noindent \verb"marco.cirant@unimi.it"\\
Dipartimento di Matematica, Universit\`a di Milano\\
via Cesare Saldini 50, 20133 Milano (Italy)
\end{flushright}

\end{document}